\documentclass{article}
\usepackage[utf8]{inputenc}

\usepackage[english]{babel}
\usepackage[shortlabels]{enumitem}
\usepackage{amsmath}
\usepackage{amsfonts}
\usepackage{amsthm}
\usepackage{amssymb}
\usepackage{mathtools}
\usepackage{tikz}
\usepackage{tikz-cd}
\usepackage[square,numbers]{natbib}
\usepackage{cleveref}
\usepackage{url}
\usepackage[T1]{fontenc}
\usetikzlibrary{fit,backgrounds}
\usepackage{spectralsequences}

\DeclareMathOperator{\Hom}{Hom}

\DeclareMathOperator{\im}{im}

\DeclareMathOperator{\Ext}{Ext}

\let\oldwidehat\widehat
\protected\def\widehat{\oldwidehat}

\theoremstyle{definition}
\newtheorem{theorem}{Theorem}
\theoremstyle{definition}
\newtheorem{corollary}{Corollary}
\theoremstyle{definition}
\newtheorem{lemma}{Lemma}
\theoremstyle{definition}
\newtheorem{proposition}{Proposition}
\theoremstyle{definition}
\newtheorem{definition}{Definition}
\theoremstyle{definition}
\newtheorem{remark}{Remark}
\theoremstyle{definition}
\newtheorem{example}{Example}
\theoremstyle{definition}

\theoremstyle{definition}
\newtheorem{construction}{Construction}
\theoremstyle{definition}
\newtheorem*{problem*}{Problem}
\theoremstyle{definition}
\newtheorem*{theorem*}{Theorem}
\theoremstyle{definition}

\title{A $v_1$-banded vanishing line for the mod 2 Moore spectrum}
\author{Kevin Chang}
\date{}

\begin{document}

\maketitle

\begin{abstract}
The mod 2 Moore spectrum $C(2)$ is the cofiber of the self-map $2: \mathbb{S} \to \mathbb{S}$. Building on work of Burklund, Hahn, and Senger, we prove that above a line of slope $\frac{1}{5}$, the Adams spectral sequence for $C(2)$ collapses at its $E_5$-page and characterize the surviving classes. This completes the proof of a result of Mahowald, announced in 1970, but never proven.
\end{abstract}

\section{Introduction}
The \emph{sphere spectrum} $\mathbb{S}$ is perhaps the fundamental object of homotopy theory, and our understanding of its homotopy groups is a good measure of our understanding of homotopy theory. The \emph{Adams spectral sequence} is one of the primary tools used to compute these groups. Computing the Adams spectral sequence for $\mathbb{S}$ is an extremely difficult problem; in fact, it is only well-understood in a finite range.

The \emph{mod 2 Moore spectrum} $C(2)$ is defined as the cofiber of the map $2: \mathbb{S} \to \mathbb{S}$. Understanding its homotopy groups and Adams spectral sequence would be a significant step towards understanding these objects for $\mathbb{S}$.

Due to \cite{adams1966periodicity}, the Adams $E_2$-page for $C(2)$ does possess nice structure in an infinite range: above a line of slope $\frac{1}{5}$, there is a periodicity isomorphism $v_1^4: E_2^{s,t}(C(2)) \xrightarrow[]{\cong} E_2^{s + 4,t + 12}(C(2))$. We will refer to the Adams spectral sequence in this range as the \emph{$v_1$-periodic Adams spectral sequence}. Understanding the $v_1$-periodic Adams spectral sequence for $C(2)$ is the first in an infinite sequence of steps towards understanding the entire structure of its Adams spectral sequence.

In \cite{mahowald1970}, Mahowald announced a theorem bounding the differentials and describing the $E_{\infty}$-page in the $v_1$-periodic Adams spectral sequence for $C(2)$. However, he never published a proof. In \cite[Proposition 15.8]{burklund2019boundaries}, Burklund, Hahn, and Senger prove a weaker bound on the differentials and the $E_{\infty}$-page than Mahowald envisioned. In this paper, we complete the proof of Mahowald's result.

Our main result is the following:

\begin{theorem}[\protect{\cite[Theorem 5]{mahowald1970}}]
\label{mahowald_thm}
Let $E_r^{s,t}(C(2))$ denote the $\mathrm{H}\mathbb{F}_2$-Adams spectral sequence for $C(2)$. Then \begin{enumerate}[(1)]
\item We have $E_5^{s,t} \cong E_{\infty}^{s,t}$ when $s \ge \frac{1}{5}(t - s) + 5$ and $t - s \ge 25$.

\item For $t - s \ge 25$, the inclusion of Adams filtrations $$F^{\frac{1}{2}(t - s) - 1.5}\pi_{t - s}C(2) \xhookrightarrow{} F^{\frac{1}{5}(t - s) + 5}\pi_{t - s}C(2)$$ is an isomorphism.
\end{enumerate}
\end{theorem}

The improvement on \cite[Proposition 15.8]{burklund2019boundaries} comes from showing that the spectral sequence collapses at $E_5$ above a vanishing line rather than $E_6$, along with the improved constants 5 and 25 for the vanishing line. The collapse at $E_5$ is optimal, as there are known to be $d_4$-differentials in the region stated. We note that \cite{mahowald1975}, without proof, claims a slightly stronger result with the expression $\frac{1}{5}(t - s) + \frac{19}{5}$ instead of $\frac{1}{5}(t - s) + 5$.

The odd-primary analogue of \Cref{mahowald_thm} follows from \cite{miller1981relations} as in \cite[\S 14]{burklund2019boundaries}. In this case, the collapse occurs on the $E_3$-page.

\subsection{Organization of this paper}
In \Cref{back_ground}, we explain some prerequisite material on synthetic spectra. \Cref{syn_intro} is a brief introduction to Pstr\k{a}gowski's category of synthetic spectra. \Cref{band_intro} introduces $v_1$-banded vanishing lines, which will allow us to reformulate \Cref{mahowald_thm} more compactly.

In \Cref{main_section}, we state and prove a stronger version of \Cref{mahowald_thm}. We also include a computation of the homotopy of the $K(1)$-local Moore spectrum, which is highly relevant to our proof of \Cref{mahowald_thm}. This computation is not available anywhere in the literature, to the best of our knowledge. \Cref{main_section} is the meat of this paper.

\subsection{Notation}
All homology $H_*$ written without coefficients is assumed to be taken with coefficients in $\mathbb{F}_2$. All Adams spectral sequences will be $\mathrm{H}\mathbb{F}_2$-Adams spectral sequences unless stated otherwise. When we have a filtered abelian group $F^{\bullet}A$ and $r \in \mathbb{R}$, $F^rA$ means $F^{\lceil r\rceil}A$. For an $\mathrm{H}\mathbb{F}_2$-nilpotent complete spectrum $X$, $F^{\bullet}\pi_*X$ will denote the $\mathrm{H}\mathbb{F}_2$-Adams filtration on the homotopy of $X$.

\subsection{Acknowledgments}
I'd like to thank Robert Burklund for suggesting and mentoring this project. I'd also like to thank Slava Gerovitch, Ankur Moitra, and David Jerison for organizing MIT's Summer Program in Undergraduate Research (SPUR), where this research was conducted. Finally, I'd like to thank the referee for providing many helpful comments.

\section{Background}
\label{back_ground}
\subsection{Synthetic spectra}
\label{syn_intro}
In order to prove facts about ordinary spectra, we will pass to the more refined $\infty$-category $\mathrm{Syn}_{\mathrm{H}\mathbb{F}_2}$, whose objects are called synthetic spectra. Rather than explain what synthetic spectra are, we will describe some of their basic properties, as developed by Pstr\k{a}gowski in \cite{pstrgowski2018synthetic}, and then explain why they are useful to our work. Our treatment will be extremely brief, so for a more detailed introduction to synthetic spectra, see \cite[\S 9, Appendix A]{burklund2019boundaries}.

\begin{construction}[Pstr\k{a}gowski]
\label{syn_exists}
There is a symmetric monoidal stable $\infty$-category $\mathrm{Syn}_{\mathrm{H}\mathbb{F}_2}$ equipped with a symmetric monoidal functor $$\nu: \mathrm{Sp} \to \mathrm{Syn}_{\mathrm{H}\mathbb{F}_2}$$ that preserves filtered colimits. Objects of $\mathrm{Syn}_{\mathrm{H}\mathbb{F}_2}$ are called \emph{$\mathrm{H}\mathbb{F}_2$-synthetic spectra}, although we will refer to them simply as \emph{synthetic spectra}.
\end{construction}

\begin{definition}[\protect{\cite[Definition 4.6, Definition 4.9]{pstrgowski2018synthetic}}]
The \emph{bigraded sphere} $\mathbb{S}^{n,n}$ is defined to be $\nu\mathbb{S}^n$. Since $\mathrm{Syn}_{\mathrm{H}\mathbb{F}_2}$ is stable, we can define the bigraded sphere $\mathbb{S}^{a,b}$ to be $\Sigma^{a - b}\mathbb{S}^{b,b}$. For any synthetic spectrum $X$, the \emph{bigraded homotopy groups} $\pi_{a,b}(X)$ are defined to be the abelian groups $\Hom(\mathbb{S}^{a,b},X)$.
\end{definition}

\begin{definition}[\protect{\cite[Definition 4.27]{pstrgowski2018synthetic}}]
Since $\Sigma$ is defined as a pushout and $\nu$ is a pointed functor, there is a canonical natural transformation $\Sigma \circ \nu \to \nu \circ \Sigma$. Applying this natural transformation to the spectrum $\mathbb{S}^{-1}$, we get a natural comparison map $$\mathbb{S}^{0,-1} \simeq \Sigma(\nu\mathbb{S}^{-1}) \to \nu(\Sigma\mathbb{S}^{-1}) \simeq \mathbb{S}^{0,0}.$$ We denote this comparison map $\tau$. We denote the cofiber of $\tau$ by $C\tau$.

A synthetic spectrum is said to be \emph{$\tau$-invertible} if the map $$\tau: \Sigma^{0,-1}X \to X$$ is an equivalence.
\end{definition}

\begin{theorem}[Pstr\k{a}gowski]
\label{tau_inv}
\begin{enumerate}[(1)]
\item The localization functor $\tau^{-1}$ is symmetric monoidal.

\item The full subcategory of $\tau$-invertible synthetic spectra is equivalent to the category of spectra.

\item The composite $\tau^{-1} \circ \nu$ is equivalent to the identity.
\end{enumerate}
\end{theorem}

\begin{remark}
Intuitively, \Cref{tau_inv} says that the only difference between ordinary spectra and synthetic spectra is the presence of $\tau$. In particular, for an ordinary spectrum $X$, the $\tau$-torsion in the bigraded homotopy groups of $\nu X$ contains information not captured by the homotopy groups of $X$. As we will see (\Cref{synth_ass}), this information is related to the differentials in the Adams spectral sequence for $X$.
\end{remark}

The next few properties begin to show the relationship between synthetic spectra and Adams spectral sequences.

\begin{lemma}[\protect{\cite[Lemma 4.23]{pstrgowski2018synthetic}}]
\label{cofib_synth}
Let $$A \to B \to C$$ be a cofiber sequence of ordinary spectra. Then $$\nu A \to \nu B \to \nu C$$ is a cofiber sequence of synthetic spectra if and only if $$0 \to H_*A \to H_*B \to H_*C \to 0$$ is a short exact sequence.
\end{lemma}

\begin{remark}
A short exact sequence $$0 \to H_*A \to H_*B \to H_*C \to 0$$ induces a long exact sequence of Adams $E_2$-pages: \begin{align*}
\cdots &\to \Ext_{A_*}^{s,t}(\mathbb{F}_2,H_*A) \to \Ext_{A_*}^{s,t}(\mathbb{F}_2,H_*B) \to \Ext_{A_*}^{s,t}(\mathbb{F}_2,H_*C) \\
&\quad \to \Ext_{A_*}^{s + 1,t}(\mathbb{F}_2,H_*A) \to \Ext_{A_*}^{s + 1,t}(\mathbb{F}_2,H_*B) \to \Ext_{A_*}^{s + 1,t}(\mathbb{F}_2,H_*C) \to \cdots.
\end{align*}
As we will see in \Cref{homo_ext}, this is the long exact sequence in homotopy associated to the cofiber sequence $$C\tau \otimes A \to C\tau \otimes B \to C\tau \otimes C.$$
\end{remark}

\begin{lemma}[\protect{\cite[Lemma 9.15]{burklund2019boundaries}}]
\label{filt_lift}
If a map $f: X \to Y$ of ordinary spectra has Adams filtration at least $k$, then there exists a factorization
\begin{center}
\begin{tikzcd}
& \Sigma^{0,-k}\nu Y \arrow[d, "\tau^k"] \\
\nu X \arrow[ur, dotted, "\widetilde{f}"] \arrow[r, "\nu f"] & \nu Y
\end{tikzcd}
\end{center}
\end{lemma}

\begin{remark}
\label{tau_div}
Although the lift $\widetilde{f}$ in \Cref{filt_lift} is not canonical in general, it is possible to pick a canonical lift for $k = 1$. Suppose $f: X \to Y$ is a map of ordinary spectra with Adams filtration at least 1; in other words, suppose it induces the 0 map on homology. If we take cofibers $$Y \to Z \to \Sigma X,$$ then it follows that $$0 \to H_*Y \to H_*Z \to H_*(\Sigma X) \to 0$$ is a short exact sequence. Hence, \Cref{cofib_synth} implies that applying $\nu$ to this sequence gives us a cofiber sequence $$\nu Y \to \nu Z \to \Sigma^{1,1}\nu X.$$ Extending this to the right and suspending gives us a map $$\widetilde{f}: \Sigma^{0,1}\nu X \to \nu Y.$$ It is shown in the proof of \cite[Lemma 9.15]{burklund2019boundaries} that $\tau\widetilde{f} = f$, so $\widetilde{f}$ is a canonical $\tau$-division of $\nu f$.
\end{remark}

\begin{example}
\label{C2_examp}
Consider the map $2: \mathbb{S}^0 \to \mathbb{S}^0$, and denote its cofiber by $C(2)$. This is the \emph{mod 2 Moore spectrum}, the main object of study in this paper. Shifting the cofiber sequence $\mathbb{S}^0 \xrightarrow[]{2} \mathbb{S}^0 \to C(2)$ to the right, we have a cofiber sequence $$\mathbb{S}^0 \to C(2) \to \mathbb{S}^1.$$ This cofiber sequence induces a short exact sequence in homology: $$0 \to H_*\mathbb{S}^0 \to H_*C(2) \to H_*\mathbb{S}^1 \to 0.$$ Thus, \Cref{cofib_synth} implies that we get a cofiber sequence of synthetic spectra $$\mathbb{S}^{0,0} \to \nu C(2) \to \mathbb{S}^{1,1}.$$ It can be shown (see \cite[Lemma 15.4]{burklund2019boundaries}) that the going-around map to the left of this cofiber sequence is a map $\widetilde{2}: \mathbb{S}^{0,1} \to \mathbb{S}^{0,0}$ such that $\nu(2) = \tau\widetilde{2}$. In particular, it turns out that $\widetilde{2}$ is the $\tau$-division provided by \Cref{tau_div}. Hence, $\nu C(2) \simeq C(\widetilde{2})$.

The existence of such a map $\widetilde{2}$ agrees with \Cref{filt_lift}, since 2 induces the 0 map on homology.
\end{example}

\begin{example}
Consider the Hopf map $\eta: \mathbb{S}^1 \to \mathbb{S}^0$. Note that $\eta$ induces the 0 map on homology. With similar reasoning to \Cref{C2_examp}, it is possible to show that $\nu C(\eta) \simeq C(\widetilde{\eta})$ for a map $\widetilde{\eta}: \mathbb{S}^{1,2} \to \mathbb{S}^{0,0}$ such that $\nu(\eta) = \tau\widetilde{\eta}$.

Since $\nu$ is symmetrical monoidal (\Cref{syn_exists}), we have an equivalence $\nu(C(2) \otimes C(\eta)) \simeq C(\widetilde{2}) \otimes C(\widetilde{\eta})$. We will denote these objects $Y \coloneqq C(2) \otimes C(\eta)$ and $\widetilde{Y} \coloneqq C(\widetilde{2}) \otimes C(\widetilde{\eta})$.
\end{example}

Despite the abstractness of the properties above, synthetic spectra are useful in a very concrete way to our work. In particular, synthetic spectra provide a clarifying perspective on Adams spectral sequences. This connection is summarized in \Cref{homo_ext}, \Cref{synth_ass}, and \Cref{synth_filt}.

\begin{lemma}[\protect{\cite[Lemma 4.56]{pstrgowski2018synthetic}}]
\label{homo_ext}
For any ordinary spectrum $X$, there is a natural isomorphism $$\pi_{t - s,t}(C\tau \otimes \nu X) \cong \Ext_{A_*}^{s,t}(\mathbb{F}_2,H_*X).$$ In other words, $\pi_{*,*}(C \tau \otimes \nu X)$ is precisely the Adams $E_2$-page of $X$.
\end{lemma}

\begin{lemma}
\label{tau_cofib}
For each nonnegative integer $i$, there is a natural cofiber sequence $$\Sigma^{0,-i}C\tau \to C\tau^{i + 1} \to C\tau^i.$$
\end{lemma}

\begin{proof}
We have a commutative diagram
\begin{center}
\begin{tikzcd}
\mathbb{S}^{0,-i - 1} \arrow[r, "\tau^{i + 1}"] \arrow[d, "\tau"] & \mathbb{S}^{0,0} \arrow[r] \arrow[d, "\simeq"] & C\tau^{i + 1} \arrow[d] \\
\mathbb{S}^{0,-i} \arrow[r, "\tau^i"] \arrow[d] & \mathbb{S}^{0,0} \arrow[r] \arrow[d] & C\tau^i \arrow[d] \\
\Sigma^{0,-i}C\tau \arrow[r] & 0 \arrow[r] & \Sigma^{1,-i}C\tau
\end{tikzcd}
\end{center}

Here, all 3 rows are cofiber sequences, and the first 2 columns are cofiber sequences. Hence, the 3rd column $$C\tau^{i + 1} \to C\tau^i \to \Sigma^{1,-i}C\tau$$ is a cofiber sequence as well. We get our desired cofiber sequence by extending this cofiber sequence to the left: $$\Sigma^{0,-i}C\tau \to C\tau^{i + 1} \to C\tau^i.$$
\end{proof}

\begin{definition}
We call the going-around map $C\tau^i \to \Sigma^{1,-i}C\tau$ from \Cref{tau_cofib} the \emph{$\tau$-Bockstein}.
\end{definition}

\begin{theorem}[\protect{\cite[Theorem 9.19]{burklund2019boundaries}}]
\label{synth_ass}
Let $X$ be an $\mathrm{H}\mathbb{F}_2$-nilpotent complete spectrum with strongly convergent Adams spectral sequence. Let $x$ denote a class in the $E_2^{s,t}$-term of the Adams spectral sequence for $X$. Then the following are equivalent: \begin{enumerate}[(1)]
\item The differentials $d_2,\ldots,d_r$ vanish on $x$.

\item $x$, as an element of $\pi_{t - s,t}(C\tau \otimes \nu X)$, lifts to $\pi_{t - s,t}(C\tau^r \otimes \nu X)$.

\item $x$ admits a lift to $\pi_{t - s,t}(C\tau^r \otimes \nu X)$ whose image under the $\tau$-Bockstein $C\tau^r \otimes \nu X \to \Sigma^{1,-r}C\tau \otimes \nu X$ is $-d_{r + 1}(x)$.
\end{enumerate}

If $x$ is a permanent cycle, then there exists a lift along the map $$\pi_{t - s,t}(\nu X) \to \pi_{t - s,t}(C\tau \otimes \nu X).$$ For such a lift $\widetilde{x}$, the following statements hold: \begin{enumerate}[(1)]
\item If $x$ survives to the $E_{r + 1}$-page, then $\tau^{r - 1}\widetilde{x} \ne 0$.

\item If $x$ survives to the $E_{\infty}$-page, then the image of $\widetilde{x}$ in $\pi_{t - s}X$ is of Adams filtration $s$ and is detected by $x$ in the Adams spectral sequence for $X$.
\end{enumerate}

Furthermore, there always exists a choice of lift $\widetilde{x}$ satisfying the following: \begin{enumerate}[(1)]
\item If $x$ is hit by a $d_{r + 1}$-differential, then we can pick $\widetilde{x}$ to be $\tau^r$-torsion.

\item If $x$ survives to the $E_{\infty}$-page and detects a class $\alpha \in \pi_{t - s}X$, then we can pick $\widetilde{x} \in \pi_{t - s,t}(\nu X)$ mapping to $\alpha$ when we invert $\tau$.
\end{enumerate}
\end{theorem}

\begin{corollary}[\protect{\cite[Corollary 9.21]{burklund2019boundaries}}]
\label{synth_filt}
Let $X$ be an $\mathrm{H}\mathbb{F}_2$-nilpotent complete spectrum with strongly convergent Adams spectral sequence. Then the decreasing filtration of $\pi_{t - s}X$ given by $$F^s\pi_{t - s}X \coloneqq \im(\pi_{t - s,t}(\nu X) \to \pi_{t - s}X)$$ coincides with the Adams filtration.
\end{corollary}

\begin{definition}
Given a synthetic spectrum $Y$, we will denote the filtration from \Cref{synth_filt} by $F^{\bullet}\pi_*(\tau^{-1}Y)$. \Cref{synth_filt} shows that this coincides with the Adams filtration on $\pi_*X$ when $Y \simeq \nu X$, so there is no notational conflict.
\end{definition}

\begin{remark}
The results above show that $\tau$-Bocksteins for synthetic spectra can be thought of as generalized Adams differentials. Even though classes in $\pi_{*,*}(C\tau \otimes Y)$ no longer correspond to classes on the $E_2$-page of an Adams spectral sequence when $Y$ is not of the form $\nu X$, it will still make sense to talk about classes in $\pi_{*,*}(C\tau \otimes Y)$ killing other classes via $\tau$-Bocksteins. One can think of these $\tau$-Bocksteins as differentials in a \emph{modified} Adams spectral sequence for $\tau^{-1}Y$. This will be important in the proof of \Cref{mahowald_thm}.
\end{remark}

We finish this section by discussing an additional property of synthetic spectra called $\tau$-completeness. This property generalizes $\mathrm{H}\mathbb{F}_2$-nilpotent completeness for ordinary spectra.

\begin{definition}[\protect{\cite[Definition A.10]{burklund2019boundaries}}]
A synthetic spectrum $X$ is \emph{$\tau$-complete} if the natural map $$X \to \varprojlim_iC\tau^i \otimes X$$ is an equivalence, where the maps $C\tau^{i + 1} \otimes X \to C\tau^i \otimes X$ are those from \Cref{tau_cofib}. In other words, $X$ is $\tau$-complete if the $\tau$-Bockstein tower is convergent.
\end{definition}

\begin{proposition}[\protect{\cite[Proposition A.13]{burklund2019boundaries}}]
\label{2_tau_comp}
Let $X$ be an ordinary spectrum. Then $X$ is $\mathrm{H}\mathbb{F}_2$-nilpotent complete if and only if $\nu X$ is $\tau$-complete.
\end{proposition}

\begin{corollary}
$C(\widetilde{2})$ and $\widetilde{Y}$ are $\tau$-complete.
\end{corollary}

\begin{proof}
Both $C(2)$ and $Y$ are $\mathrm{H}\mathbb{F}_2$-nilpotent complete (see \cite[Lemma 2.1.15]{ravenel2003complex}). Since $C(\widetilde{2}) \simeq \nu C(2)$ and $\widetilde{Y} \simeq \nu Y$, \Cref{2_tau_comp} implies that $C(\widetilde{2})$ and $\widetilde{Y}$ are $\tau$-complete.
\end{proof}

\begin{proposition}
\label{lim_tau_comp}
Limits of $\tau$-complete synthetic spectra are $\tau$-complete.
\end{proposition}

\begin{proof}
Since $C\tau^i$ is dualizable for each $i$ (see \cite[Remark 9.6]{burklund2019boundaries}), the functor $C\tau^i \otimes -$ commutes with limits. Then the proposition follows from the fact that limits commute with limits.
\end{proof}

\subsection{$v_1$-banded vanishing lines}
\label{band_intro}
The statement of our main result (\Cref{mahowald_thm}) is quite unwieldy. We would like to prove the existence of a line in the Adams spectral sequence for $C(2)$ above which the Adams spectral sequence is well-behaved in sufficiently high topological degree. Being ``well-behaved'' consists of the following key qualities: \begin{enumerate}[(1)]
\item Above some line of slope less than $\frac{1}{2}$, the spectral sequence collapses at some finite page.

\item The only classes that survive to the $E_{\infty}$-page live inside a band bounded by lines of slope $\frac{1}{2}$. These classes are \emph{$v_1$-periodic}.
\end{enumerate}

In order to prove our main result, we will want to analyze similar behavior in the bigraded homotopy of synthetic spectra that do not come from ordinary spectra. For these synthetic spectra, it is not possible to describe this behavior in terms of Adams spectral sequences. We make this phenomenon precise by using the following definition.

\begin{definition}[\protect{\cite[Definition 13.6]{burklund2019boundaries}}]
Let $X$ be a synthetic spectrum. We say that $X$ has a \emph{$v_1$-banded vanishing line} with \begin{itemize}
\item band intercepts $b \le d$

\item range of validity $v$

\item line of slope $m < \frac{1}{2}$ and intercept $c$

\item torsion bound $r$
\end{itemize}
if the following conditions hold: \begin{enumerate}[(1)]
\item Every $\tau$-power torsion class in $\pi_{t - s,t}X$ is $\tau^r$-torsion for $s \ge m(t - s) + c$ and $t - s \ge v$.

\item The natural map $$F^{\frac{1}{2}(t - s) + b}\pi_{t - s}(\tau^{-1}X) \to F^{m(t - s) + c}\pi_{t - s}(\tau^{-1}X)$$ is an isomorphism for $t - s \ge v$.

\item The composite $$F^{\frac{1}{2}(t - s) + b}\pi_{t - s}(\tau^{-1}X) \to \pi_{t - s}(\tau^{-1}X) \to \pi_{t - s}(L_{K(1)}\tau^{-1}X)$$ is an isomorphism for $t - s \ge v$.

\item $\pi_{t - s,t}X \cong 0$ for $s > \frac{1}{2}(t - s) + d$.
\end{enumerate}
More concisely, we will say that $X$ has a $v_1$-banded vanishing line with parameters $(b \le d,v,m,c,r)$.
\end{definition}

\begin{remark}
Given an $\mathrm{H}\mathbb{F}_2$-nilpotent complete ordinary spectrum $X$, we will say that the Adams spectral sequence of $X$ has a $v_1$-banded vanishing line with parameters $(b \le d,v,m,c,r)$ when the synthetic spectrum $\nu X$ has one.
\end{remark}

\begin{figure}
\centering
\includegraphics[scale=0.7]{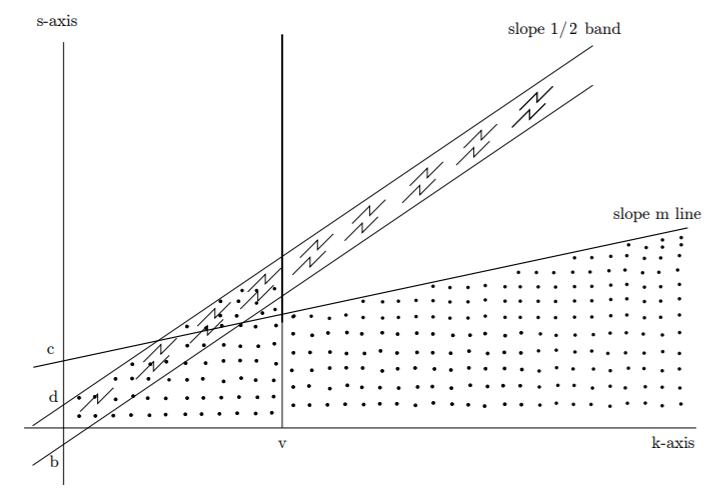}
\caption{A cartoon of the $E_{r + 1}$-page of the Adams spectral sequence for a $\mathrm{H}\mathbb{F}_2$-nilpotent complete spectrum $X$ admitting a $v_1$-banded vanishing line with parameters $(b \le d,v,m,c,r)$. Here, $k \coloneqq t - s$ is the topological degree. This picture is copied from Figure 1 in \cite{burklund2019boundaries}.}\label{v1_band}
\end{figure}

\section{A $v_1$-banded vanishing line for $C(2)$}
\begin{figure}[ht!]
\centering
\includegraphics[scale=0.7]{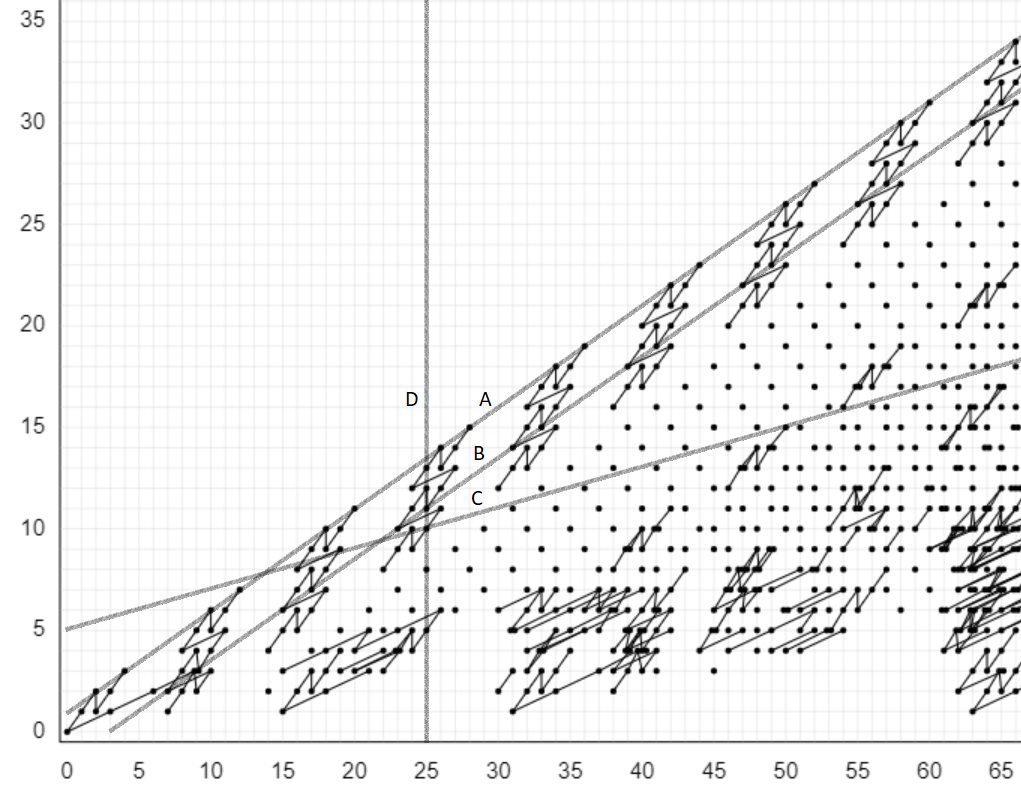}
\caption{The $E_2$-page of the Adams spectral sequence for $C(2)$. We will prove that all the classes in between lines B and C and to the right of line D vanish because of the differentials $d_2,d_3,d_4$. Moreover, it will follow from our $v_1$-banded vanishing line and the computation of the cardinality of $\pi_iL_{K(1)}C(2)$ that the surviving classes above line $C$ and to the right of line $D$ are precisely the ``lightning flashes'' between lines A and B. This image was created using Hood Chatham's Adams spectral sequence calculator, which can be found on his website.}\label{C2_ASS}
\end{figure}

\label{main_section}
We recall the statement of our main result, \Cref{mahowald_thm}.

\begin{theorem*}
Let $E_r^{s,t}(C(2))$ denote the $\mathrm{H}\mathbb{F}_2$-Adams spectral sequence for $C(2)$. Then \begin{enumerate}[(1)]
\item There exist constants $c$ and $v$ such that $E_5^{s,t} \cong E_{\infty}^{s,t}$ when $s \ge \frac{1}{5}(t - s) + c$ and $t - s \ge v$.

\item We can pick the constants $c$ and $v$ in part (1) such that for $t - s \ge v$, the inclusion of Adams filtrations $$F^{\frac{1}{2}(t - s) - 1.5}\pi_{t - s}C(2) \xhookrightarrow{} F^{\frac{1}{5}(t - s) + c}\pi_{t - s}C(2)$$ is an isomorphism.
\end{enumerate}
\end{theorem*}

In other words, part (1) of \Cref{mahowald_thm} states that the $v_1$-periodic Adams spectral sequence for $C(2)$ has no differentials past $d_2,d_3,d_4$. Part (2) states that the only surviving classes above this line fall between the lines $s = \frac{1}{2}(t - s) - 1.5$ and $s = \frac{1}{2}(t - s) + 1$.

In \cite{burklund2019boundaries}, Burklund, Hahn, and Senger prove part (2) of \Cref{mahowald_thm} and come close to a proof of part (1). The authors prove that the synthetic spectrum $C(\widetilde{2})$ has a $v_1$-banded vanishing line with parameters $$(b \le d,v,m,c,r) = \left(-3.5 \le 1,28 + \frac{1}{3},0.2,5,4\right).$$

We will complete the proof of \Cref{mahowald_thm} by proving the following stronger result:
\begin{theorem}
\label{mahowald_thm_better}
The synthetic spectrum $C(\widetilde{2})$ admits a $v_1$-banded vanishing line with parameters $$(b \le d,v,m,c,r) = \left(-1.5 \le 1,25,0.2,5,3\right).$$
\end{theorem}

\subsection{The $K(1)$-local Moore spectrum}
\label{K1_local_C2}
We will use the following information about the homotopy groups of $L_{K(1)}C(2)$. This information can be obtained by examining the long exact sequence in homotopy associated to the cofiber sequence $L_{K(1)}\mathbb{S} \xrightarrow[]{2} L_{K(1)}\mathbb{S} \to L_{K(1)}C(2)$. While the homotopy of $L_{K(1)}\mathbb{S}$ is well-known, the computation is not readily available in the literature, so we provide one below. We use the following characterization of $L_{K(1)}\mathbb{S}$ as a homotopy fixed point spectrum.

\begin{theorem}[\protect{\cite[Proposition 7.1]{DEVINATZ20041}}]
\label{K1_sphere}
Let $\widehat{KU}$ denoted 2-completed complex $K$-theory. This spectrum carries an action of the Morava stabilizer group $\mathbb{G}_1 \cong \mathbb{Z}_2^{\times}$, where for $i \in \mathbb{Z}_2^{\times}$ the action $$\psi^i: \widehat{KU} \to \widehat{KU}$$ is the classical $i$th Adams operation. Then the $K(1)$-local sphere $L_{K(1)}\mathbb{S}$ is homotopy equivalent to the homotopy fixed point spectrum $\widehat{KU}^{h\mathbb{Z}_2^{\times}}$.
\end{theorem}

\begin{corollary}
\label{K1_sphere_hom}
The $K(1)$-local sphere $L_{K(1)}\mathbb{S}$ has the following homotopy groups: $$\pi_iL_{K(1)}\mathbb{S} \cong \begin{cases}\mathbb{Z}_2 & \text{if }i = -1 \\ \mathbb{Z}/2^{4 + v} & \text{if }i = 8k - 1,k = 2^vm,m\text{ odd} \\ \mathbb{Z}/2 \oplus \mathbb{Z}_2 & \text{if }i = 0 \\ \mathbb{Z}/2 & \text{if }i = 8k,k \ne 0 \\ \mathbb{Z}/2 \oplus \mathbb{Z}/2 & \text{if }i = 8k + 1 \\ \mathbb{Z}/2 & \text{if }i = 8k + 2 \\ \mathbb{Z}/8 & \text{if }i = 8k + 3 \\ 0 & \text{if }i \in \{8k + 4,8k + 5,8k + 6\}.\end{cases}$$
Here, $\mathbb{Z}_2$ denotes the 2-adic integers.
\end{corollary}

\begin{proof}[Proof sketch]
There is a splitting of topological groups $\mathbb{Z}_2^{\times} \cong \{\pm 1\} \times (1 + 4\mathbb{Z}_2)^{\times}$, so \Cref{K1_sphere} implies that $$L_{K(1)}\mathbb{S} \simeq \widehat{KU}^{h\mathbb{Z}_2^{\times}} \simeq \left(\widehat{KU}^{h\{\pm 1\}}\right)^{h(1 + 4\mathbb{Z}_2)^{\times}} \simeq \widehat{KO}^{h(1 + 4\mathbb{Z}_2)^{\times}},$$ where $\widehat{KO}$ is 2-completed real K-theory. The fact that $KO \simeq KU^{hC_2}$ is well-known to follow from \cite{atiyah1966k} and a computation using a homotopy fixed point spectral sequence (see \cite[Proposition 5.3.1]{rognes2008} for a proof).

The group $(1 + 4\mathbb{Z}_2)^{\times}$ has 5 as a topological generator, so there is a cofiber sequence $$\widehat{KO}^{h(1 + 4\mathbb{Z}_2)^{\times}} \to \widehat{KO} \xrightarrow[]{1 - \psi^5} \widehat{KO}.$$ We can therefore use the long exact sequence in homotopy groups associated to the cofiber sequence $$L_{K(1)}\mathbb{S} \to \widehat{KO} \xrightarrow[]{1 - \psi^5} \widehat{KO}$$ to compute $\pi_*(L_{K(1)}\mathbb{S})$.

This finishes the proof of the theorem for all $i$ except $i \equiv 1 \pmod{8}$, in which case there is an extension problem $$0 \to \mathbb{Z}/2 \to \pi_iL_{K(1)}\mathbb{S} \to \mathbb{Z}/2 \to 0.$$ We can deduce that $\pi_i(L_{K(1)}\mathbb{S}) \cong \mathbb{Z}/2 \oplus \mathbb{Z}/2$ by examining the Adams spectral sequence for $\mathbb{S}$.

The following pattern is present in the $K(1)$-local part of the Adams spectral sequence for the sphere (see \cite[Corollary 1.3]{DAVIS198939}). For each integer $k > 0$, the dotted line indicates a possible hidden 2-extension in $\pi_{8k + 1}L_{K(1)}\mathbb{S}$ from the class in $(8k + 1,4k)$ to the class in $(8k + 1,4k + 1)$. The diagram below shows the pattern for $k = 3$.

\begin{center}
\begin{sseqdata}[ name = K1_S_ASS, Adams grading, classes = fill ]
\class["h_1P^2c_0" right](25,12)
\class["P^3h_1" right](25,13)
\class(26,14)
\class["P^3h_2" below](27,13)
\class(27,14)
\class(27,15)
\structline[dotted](25,12)(25,13)
\structline(25,13)(26,14)
\structline(26,14)(27,15)
\structline(27,13)(27,14)
\structline(27,14)(27,15)
\end{sseqdata}
\printpage[ name = K1_S_ASS, page = 2 ]
\end{center}
\end{proof}

We claim that this hidden 2-extension from $h_1P^{k - 1}c_0$ to $P^kh_1$ does not occur. Suppose it does. Then we have a relation in homotopy $$[P^kh_1] = 2[h_1P^{k - 1}c_0].$$ Then $$\eta[P^kh_1] = 2\eta[h_1P^{k - 1}c_0] = 0,$$ a contradiction. This shows that no extension occurs, and we have $$\pi_{8k + 1}L_{K(1)}\mathbb{S} \cong \mathbb{Z}/2 \oplus \mathbb{Z}/2.$$

\begin{corollary}
\label{hom_moore}
The homotopy groups of $L_{K(1)}C(2)$ have the following orders:

\begin{center}
\begin{tabular}{|c|c|}
\hline
$i \pmod{8}$ & $\left|\pi_i(L_{K(1)}C(2))\right|$ \\
\hline
0 & 4 \\
\hline
1 & 8 \\
\hline
2 & 8 \\
\hline
3 & 4 \\
\hline
4 & 2 \\
\hline
5 & 1 \\
\hline
6 & 1 \\
\hline
7 & 2 \\
\hline
\end{tabular}
\end{center}
\end{corollary}

\begin{proof}
This follows from considering the long exact sequence in homotopy associated to the cofiber sequence $$L_{K(1)}\mathbb{S} \xrightarrow[]{2} L_{K(1)}\mathbb{S} \to L_{K(1)}C(2).$$ This is a cofiber sequence because localization preserves cofiber sequences.
\end{proof}

There are a few extension problems to solve before we know the group structure of $\pi_*L_{K(1)}C(2)$. They are the following:
$$0 \to \mathbb{Z}/2 \to \pi_{8k}L_{K(1)}C(2) \to \mathbb{Z}/2 \to 0\quad (k \ne 0)$$
$$0 \to \mathbb{Z}/2 \oplus \mathbb{Z}/2 \to \pi_{8k + 1}L_{K(1)}C(2) \to \mathbb{Z}/2 \to 0$$
$$0 \to \mathbb{Z}/2 \to \pi_{8k + 2}L_{K(1)}C(2) \to \mathbb{Z}/2 \oplus \mathbb{Z}/2 \to 0$$
$$0 \to \mathbb{Z}/2 \to \pi_{8k + 3}L_{K(1)}C(2) \to \mathbb{Z}/2 \to 0$$
Here, the extension for $8k$ is trivial for $k = 0$, since \Cref{K1_sphere_hom} implies that $\pi_0L_{K(1)}C(2) \cong \mathbb{Z}/2 \oplus \mathbb{Z}/2$.

That said, we will need only the orders of these groups, not the groups themselves, to prove \Cref{mahowald_thm}. After proving \Cref{mahowald_thm}, however, we will be able to provide full descriptions of the homotopy groups.

\subsection{Proof of \Cref{mahowald_thm}}
In \cite{burklund2019boundaries}, the authors deduce a $v_1$-banded vanishing line for $C(\widetilde{2})$ from a $v_1$-banded vanishing line for $\widetilde{Y} \coloneqq C(\widetilde{2}) \otimes C(\widetilde{\eta})$. This last fact translates to a $v_1$-banded vanishing line in the Adams spectral sequence for $C(2) \otimes C(\eta)$, since $\widetilde{Y} \simeq \nu(C(2) \otimes C(\eta))$ (see \cite[Lemma 15.4]{burklund2019boundaries}). The authors then use the following lemma to go from $\widetilde{Y}$ to $C(\widetilde{2})$ through cofiber sequences.

\begin{lemma}[\protect{\cite[Proposition 13.11]{burklund2019boundaries}}]
\label{cofib_line}
Let $A \to B \to C$ be a cofiber sequence of synthetic spectra such that the following conditions hold: \begin{itemize}
\item $A$ has a $v_1$-banded vanishing line with parameters $(b_A \le d_A,v_A,m,c_A,r_A)$.

\item $C$ has a $v_1$-banded vanishing line with parameters $(b_C \le d_C,v_C,m,c_C,r_C)$.
\end{itemize}

Then $B$ has a $v_1$-banded vanishing line with parameters $(b_B \le d_B,v_B,m,c_B,r_B)$, where \begin{align*}
b_B &= \min(b_A,b_C - r_A) \\
d_B &= \max(d_A,d_C) \\
v_B &= \max\left(v_A + 1,v_C,\frac{c_B - b_B}{1/2 - m}\right) \\
c_B &= \max(c_A + r_A,c_C) \\
r_B &= r_A + \max\left(r_C,\left\lfloor\max(d_A,\min(d_A + r_C,d_C)) - b_C - \frac{1}{2}\right\rfloor\right).
\end{align*}
\end{lemma}

In \cite{burklund2019boundaries}, the authors start with the following theorem, which is proven using the computation of the $E_2$-page of the $v_1$-periodic Adams spectral sequence of $Y = C(2) \otimes C(\eta)$ in \cite{davis1988v_1} and Miller's method for computing differentials in \cite{miller1981relations}.

\begin{theorem}[\protect{\cite[Theorem 14.1]{burklund2019boundaries}}]
\label{Y_vanish}
The synthetic spectrum $\widetilde{Y} = C(\widetilde{2}) \otimes C(\widetilde{\eta})$ has a $v_1$-banded vanishing line with parameters $(-1.5 \le 0,15,0.2,2.6,1)$.
\end{theorem}

The authors then apply \Cref{cofib_line} to the following cofiber sequences (see \cite[Proposition 15.8]{burklund2019boundaries}): \begin{align}
\Sigma^{1,2}C(\widetilde{2}) \otimes C(\widetilde{\eta}) \to C(\widetilde{2}) \otimes C(\widetilde{\eta}^2) \to C(\widetilde{2}) \otimes C(\widetilde{\eta}) \label{cofib_Y2}\\
\Sigma^{1,2}C(\widetilde{2}) \otimes C(\widetilde{\eta}^2) \to C(\widetilde{2}) \otimes C(\widetilde{\eta}^3) \to C(\widetilde{2}) \otimes C(\widetilde{\eta}) \label{cofib_Y3_1}
\end{align}
Finally, they use the splitting (see \cite[Lemma 15.7]{burklund2019boundaries}) $$C(\widetilde{2}) \otimes C(\widetilde{\eta}^3) \simeq C(\widetilde{2}) \oplus \Sigma^{4,6}C(\widetilde{2})$$ to deduce a $v_1$-banded vanishing line for $C(\widetilde{2})$.

The synthetic spectra appearing in the above cofiber sequences will be important to us. We define $\widetilde{Y}_2 \coloneqq C(\widetilde{2}) \otimes C(\widetilde{\eta}^2)$ and $\widetilde{Y}_3 \coloneqq C(\widetilde{2}) \otimes C(\widetilde{\eta}^3)$.

The main content of \Cref{mahowald_thm} lies in decreasing $r_{\widetilde{Y}_3}$. We will do so by increasing $b_{\widetilde{Y}_2}$ and then applying \Cref{cofib_line} to the cofiber sequence (\ref{cofib_Y3_1}). Since the formula for $r_B$ in \Cref{cofib_line} contains $-b_C$, raising $b_{\widetilde{Y}_2}$ will allow us to lower $r_{\widetilde{Y}_3}$.

In concrete terms, this means raising the bottom edges of the bands obtained for $\widetilde{Y}_2$ and $\widetilde{Y}_3$ in \cite{burklund2019boundaries}. This entails showing that classes on and slightly above these bottom edges are killed by differentials. We prove that these differentials are nonzero by examining all possible classes that occur and exploiting previously obtained vanishing lines to show that these classes must be killed.

We first deduce a $v_1$-banded vanishing line for $\widetilde{Y}_2$. Applying \Cref{cofib_line} and \Cref{Y_vanish} yields a $v_1$-banded vanishing line for $\widetilde{Y}_2$ with parameters $(-2.5 \le 0.5,23,0.2,4.4,2)$. We wish to raise $b$ from $-2.5$ to $-1.5$.

\begin{proposition}
\label{Y2_vanish}
$\widetilde{Y}_2$ admits a $v_1$-banded vanishing line with parameters $(-1.5 \le 0.5,23,0.2,4.4,2)$.
\end{proposition}

\begin{proof}
It is enough to show that for $$\frac{1}{2}(t - s) - 2.5 \le s < \frac{1}{2}(t - s) - 1.5$$ with $t - s \ge 23$, the classes in $\pi_{t - s,t}(C\tau \otimes \widetilde{Y}_2)$ that lift to $\pi_{t - s,t}\widetilde{Y}_2$ (we will call these \emph{permanent cycles}) have lifts that are killed by some power of $\tau$ (\emph{eventual boundaries}). Since we already have a $v_1$-banded vanishing line with $r = 2$, in fact all eventual boundaries will lift to $\tau^2$-torsion.

Our terminology comes from the relationship between bigraded homotopy groups and Adams spectral sequences. \Cref{synth_ass} is a precise formulation of this relationship for synthetic spectra obtained by applying $\nu$ to an ordinary spectrum. For $\widetilde{Y}_2$, which is not $\nu$ applied to an ordinary spectrum, it may be helpful to think of $\pi_{t - s,t}(C\tau \otimes \widetilde{Y}_2)$ as the $E_2^{s,t}$-term of a \emph{modified} Adams spectral sequence for $\tau^{-1}\widetilde{Y}_2$, if the reader is familiar with this notion.

\begin{figure}
\centering
\begin{sseqdata}[ name = Y_ASS, Adams grading, x range = {28}{32}, y range = {12}{15} ]
\class(28,14)
\class(29,13)
\class["a" left](30,12)
\class(30,15)
\class(31,14)
\class(32,13)
\structline(29,13)(31,14)

\class(27,12)
\structline(27,12)(29,13)

\class(33,15)
\structline(31,14)(33,15)
\end{sseqdata}
\printpage[ name = Y_ASS, page = 2 ]
\caption{The bigraded homotopy groups $\pi_{t - s,t}(C\tau \otimes \widetilde{Y})$. The line is the bottom edge of the band of \Cref{Y_vanish}. The exact homotopy groups are computed near the band in \cite[\S 14]{burklund2019boundaries}.}\label{Y_ASS_ss1}
\end{figure}

\begin{figure}
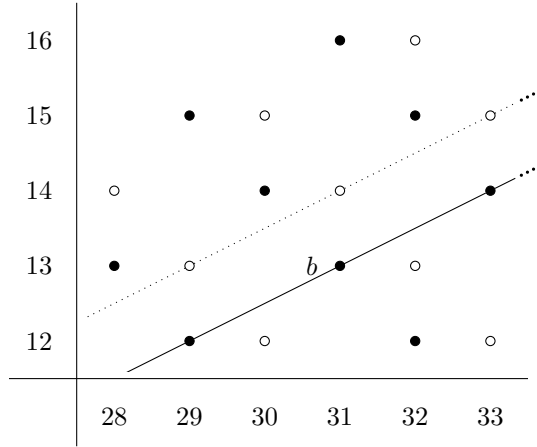

\centering
\begin{sseqdata}[ name = Y2_ASS, Adams grading, x range = {28}{33}, y range = {12}{16} ]
\class(28,14)
\class(29,13)
\class(30,12)
\class(30,15)
\class(31,14)
\class(32,13)
\class(29,15)
\class(30,14)
\class(31,13)
\class(31,16)
\class(32,15)
\class(33,14)
\class(32,12)
\class(33,12)
\class(28,13)
\class(29,12)
\class(32,16)
\class(33,15)
\classoptions[fill, black](29,15)
\classoptions[fill, black](30,14)
\classoptions["b" left, fill, black](31,13)
\classoptions[fill, black](31,16)
\classoptions[fill, black](32,15)
\classoptions[fill, black](33,14)
\classoptions[fill, black](32,12)
\classoptions[fill, black](28,13)
\classoptions[fill, black](29,12)
\structline(29,12)(33,14)
\structline[dotted](29,13)(31,14)
\structline[dotted](31,14)(33,15)

\class(27,12)
\structline[dotted](27,12)(29,13)

\class(35,16)
\structline[dotted](33,15)(35,16)

\class(27,11)
\structline(27,11)(29,12)

\class(35,15)
\structline(33,14)(35,15)
\end{sseqdata}
\printpage[ name = Y2_ASS, page = 2 ]
\caption{An upper bound on $\pi_{t - s,t}(C\tau \otimes \widetilde{Y}_2)$. Black dots indicate possible classes in the image of $\pi_{*,*}(\Sigma^{1,2}C\tau \otimes \widetilde{Y}) \to \pi_{*,*}(C\tau \otimes \widetilde{Y}_2)$, whereas the white dots indicate possible classes that have nonzero images along the map $\pi_{*,*}(C\tau \otimes \widetilde{Y}_2) \to \pi_{*,*}(C\tau \otimes \widetilde{Y})$. The solid line is the bottom edge of our initial band, and we want to raise it to the dotted line.}\label{Y2_ASS_ss1}
\end{figure}

Figures \ref{Y_ASS_ss1} and \ref{Y2_ASS_ss1} require some explanation. A dot in position $(t - s,s)$ corresponds to a copy of $\mathbb{F}_2$ in the bigraded homotopy group $\pi_{t - s,t}(C\tau \otimes -)$. \Cref{Y_ASS_ss1} is the bigraded homotopy of $C\tau \otimes \widetilde{Y}$ in this range. Meanwhile, \Cref{Y2_ASS_ss1} is only an upper bound obtained from the long exact sequence associated to the cofiber sequence $$\Sigma^{1,2}C\tau \otimes \widetilde{Y} \to C\tau \otimes \widetilde{Y}_2 \to C\tau \otimes \widetilde{Y}.$$

We obtain \Cref{Y2_ASS_ss1} from two copies of \Cref{Y_ASS_ss1}, with the unfilled dots being unmoved and the filled dots being shifted by $(1,1)$ (since they come from $\Sigma^{1,2}C\tau \otimes \widetilde{Y} \to C\tau \otimes \widetilde{Y}_2$).

Because $Y$ admits a $v_1$-self map \cite[Theorem 1.2]{davis1981v} and the region of interest lies in the $v_1$-periodic region of $\pi_{*,*}$ \cite[Lemma 14.21]{burklund2019boundaries}, the classes in $\pi_{*,*}(C\tau \otimes \widetilde{Y})$ that we are concerned with and, therefore, the upper bound on $\pi_{*,*}(C\tau \otimes \widetilde{Y}_2)$ (see \Cref{Y2_ASS_ss1}) are $v_1$-periodic, so they repeat in increments of $(1,2)$. Hence, we will only draw them in the indicated range.

Observe that the only possible nonzero homotopy groups in the range $$\frac{1}{2}(t - s) - 2.5 \le s < \frac{1}{2}(t - s) - 1.5$$ with $t - s \ge 23$ are $\pi_{2k + 1,(2k + 1) + (k - 2)}(C\tau \otimes \widetilde{Y}_2)$ for $k \ge 11$ and that these groups are either $\mathbb{F}_2$ or 0. Figures \ref{Y_ASS_ss1} and \ref{Y2_ASS_ss1} illustrate the case $k = 15$.

For brevity, let $i \coloneqq 2k + 1$ and $j \coloneqq k - 2$. In Figures \ref{Y_ASS_ss1} and \ref{Y2_ASS_ss1}, we have $i = 31$ and $j = 13$.

If $\pi_{i,i + j}(C\tau \otimes \widetilde{Y}) \cong 0$, then there exists no permanent cycle in $\pi_{i,i + j}(C\tau \otimes \widetilde{Y}_2)$, so there is nothing to check.

Now suppose $\pi_{i,i + j}(C\tau \otimes \widetilde{Y}_2) \cong \mathbb{F}_2$ with nonzero element $b$. Then by exactness, there is a unique element $$a \in \pi_{i,i + j}(\Sigma^{1,2}C\tau \otimes \widetilde{Y})$$ mapping to $b$ when we smash the cofiber sequence (\ref{cofib_Y2}) with $C\tau$. Because $\widetilde{Y}$ has a $v_1$-banded vanishing line with parameters $(-1.5 \le 0,15,0.2,2.6,1)$, $\Sigma^{1,2}\widetilde{Y}$ has a $v_1$-banded vanishing line with parameters $(-1 \le 0.5,16,0.2,3.4,1)$. Applying \Cref{synth_ass} to $Y$ shows that $a$ is a permanent cycle, since all the Adams differentials out of $(2k + 1,k - 2)$ are 0 for degree reasons. Since $k \ge 11$, we have $$\frac{1}{5}(2k + 1) + 3.4 \le k - 2 < \frac{1}{2}(2k + 1).$$ This means that $a$ lies in the vanishing region of $\widetilde{Y}$. As a permanent cycle in the vanishing region, $a$ lifts to a $\tau$-torsion class $\hat{a}\in \pi_{i,i + j}(\Sigma^{1,2}\widetilde{Y})$. Then the image of $\hat{a}$ under the map $$\pi_{i,i + j}(\Sigma^{1,2}\widetilde{Y}) \to \pi_{i,i + j}\widetilde{Y}_2$$ is a $\tau$-torsion class lifting $b$.

In the language of modified Adams spectral sequences, $b$ is killed by a $d_2$ induced by a $d_2$ in the Adams spectral sequence for $Y$.

Since there are no other possible permanent cycles in the range $$\frac{1}{2}(t - s) - 2.5 \le s < \frac{1}{2}(t - s) - 1.5$$ with $t - s \ge 23$, we have shown that $\widetilde{Y}_2$ admits a $v_1$-banded vanishing line with parameters $(-1.5 \le 0.5,23,0.2,4.4,2)$.
\end{proof}

\Cref{Y2_vanish} is good enough to establish the key part of Mahowald's theorem, which is that there exist no differentials higher than $d_4$ in the Adams spectral sequence for $C(2)$ above a line of the form $s = \frac{1}{5}(t - s) + c$ (i.e. $r_{\widetilde{Y}} = 3$). We can see this by applying \Cref{cofib_line} to the cofiber sequence $$\Sigma^{2,4}\widetilde{Y} \to \widetilde{Y}_3 \to \widetilde{Y}_2.$$

However, we will go further. In the following proof, we use similar methods to show an improved $v_1$-banded vanishing line for $\widetilde{Y}_3$. Since $\widetilde{Y}_3$ splits into copies of $C(\widetilde{2})$, this will completely prove \Cref{mahowald_thm_better} and thus \Cref{mahowald_thm}.

\begin{proof}[Proof of \Cref{mahowald_thm_better}]
Applying \Cref{cofib_line}, \Cref{Y_vanish}, and \Cref{Y2_vanish} to the cofiber sequence \begin{align}
\Sigma^{2,4}\widetilde{Y} \to \widetilde{Y}_3 \to \widetilde{Y}_2 \label{cofib_Y3_2},
\end{align}
we see that $\widetilde{Y}_3$ admits a $v_1$-banded vanishing line with parameters $(-2.5 \le 1,29,0.2,6.2,3)$. All that remains is to raise $b$ from $-2.5$ to $-1.5$, as the splitting $\widetilde{Y}_3 \simeq C(\widetilde{2}) \oplus \sigma^{4,6}C(\widetilde{2})$ will give us the desired $v_1$-banded vanishing line for $C(\widetilde{2})$.

Now recall the cofiber sequence (\ref{cofib_Y3_1}): $$\Sigma^{1,2}\widetilde{Y}_2 \to \widetilde{Y}_3 \to \widetilde{Y}.$$ We again get upper bounds on the homotopy of $C\tau \otimes \widetilde{Y}_3$ by exactness.

\begin{figure}
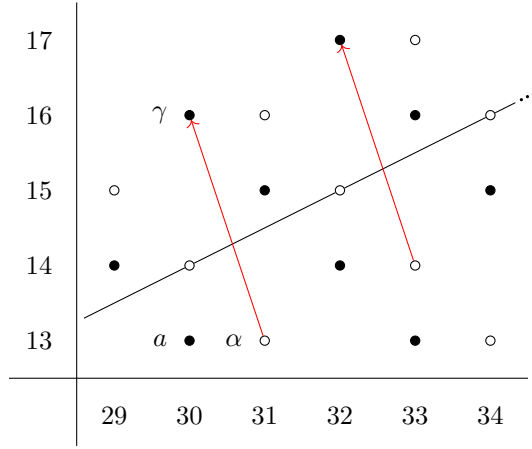

\centering
\begin{sseqdata}[ name = Y2_ASS_2, Adams grading, x range = {29}{34}, y range = {13}{17} ]
\class(29,15)
\class(30,14)
\class["\alpha" left](31,13)
\class(31,16)
\class(32,15)
\class(33,14)
\class["\gamma" left](30,16)
\class(31,15)
\class(32,14)
\class(32,17)
\class(33,16)
\class(34,15)
\class(33,13)
\class(34,13)
\class(29,14)
\class["a" left](30,13)
\class(33,17)
\class(34,16)
\classoptions[fill, black](30,16)
\classoptions[fill, black](31,15)
\classoptions[fill, black](32,14)
\classoptions[fill, black](32,17)
\classoptions[fill, black](33,16)
\classoptions[fill, black](34,15)
\classoptions[fill, black](33,13)
\classoptions[fill, black](29,14)
\classoptions[fill, black](30,13)
\structline(30,14)(32,15)
\structline(32,15)(34,16)

\d[red]3(31,13)
\d[red]3(33,14)

\class(28,13)
\structline(28,13)(30,14)

\class(36,17)
\structline(34,16)(36,17)
\end{sseqdata}
\printpage[ name = Y2_ASS_2, page = 3 ]
\caption{An upper bound on $\pi_{t - s,t}(\Sigma^{1,2}C\tau \otimes \widetilde{Y}_2)$. Black and white dots mean the same thing as they did in \Cref{Y2_ASS_ss1}. The line is the bottom edge of the band from \Cref{Y2_vanish}. We will show that the red differentials are 0.}\label{Y2_ASS_ss2}
\end{figure}

\begin{figure}
\centering
\begin{sseqdata}[ name = Y3_ASS_2, Adams grading, x range = {28}{34}, y range = {12}{17} ]
\class(29,15)
\class(30,14)
\class["\beta" left](31,13)
\class(31,16)
\class(32,15)
\class(33,14)
\class(30,16)
\class(31,15)
\class(32,14)
\class(32,17)
\class(33,16)
\class(34,15)
\class(33,13)
\class(34,13)
\class(29,14)
\class["b" left](30,13)
\class(33,17)
\class(34,16)
\class(31,12)
\class(33,12)
\class(34,12)
\class(32,12)
\class(34,12)
\class(28,13)
\class(29,12)
\class(28,12)
\class(28,15)
\classoptions[fill, black](30,16)
\classoptions[fill, black](31,15)
\classoptions[fill, black](32,14)
\classoptions[fill, black](32,17)
\classoptions[fill, black](33,16)
\classoptions[fill, black](34,15)
\classoptions[fill, black](33,13)
\classoptions[fill, black](29,14)
\classoptions[fill, black](30,13)
\classoptions[fill, black](31,12)
\classoptions[fill, black](33,12)
\classoptions[fill, black](34,12)
\classoptions[fill, black](28,12)
\classoptions[fill, black](28,15)

\class[rectangle](28,14)
\class[rectangle](29,13)
\class[rectangle](30,12)
\class[rectangle](30,15)
\class[rectangle](31,14)
\class[rectangle](32,13)
\class[rectangle](33,12)
\class[rectangle](32,16)
\class[rectangle](33,15)
\class[rectangle](34,14)
\class[rectangle](34,17)

\structline(29,12)(31,13)
\structline(31,13)(33,14)
\structline[dotted](29,13)(31,14)
\structline[dotted](31,14)(33,15)

\class(27,12)
\structline[dotted](27,12)(29,13)

\class(35,16)
\structline[dotted](33,15)(35,16)

\class(27,11)
\structline(27,11)(29,12)

\class(35,15)
\structline(33,14)(35,15)
\end{sseqdata}
\printpage[ name = Y3_ASS_2, page = 3 ]
\caption{An upper bound on $\pi_{t - s,t}(C\tau \otimes \widetilde{Y}_3)$. Dots indicate possible classes in the image of $\pi_{*,*}(\Sigma^{1,2}C\tau \otimes \widetilde{Y}_2) \to \pi_{*,*}(C\tau \otimes \widetilde{Y}_3)$, whereas squares indicate possible classes that have nonzero images along the map $\pi_{*,*}(C\tau \otimes \widetilde{Y}_3) \to \pi_{*,*}(C\tau \otimes \widetilde{Y})$. The solid line is the bottom edge of our initial band, and we want to raise it to the dotted line.}\label{Y3_ASS_ss2}
\end{figure}

\begin{figure}
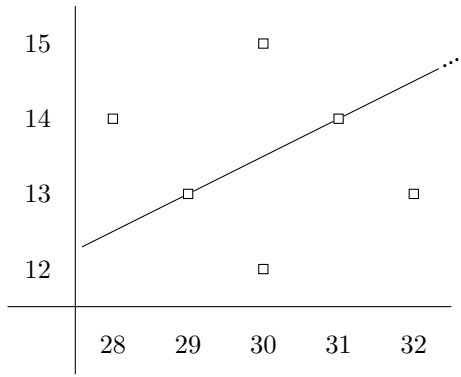

\centering
\begin{sseqdata}[ name = Y_ASS_2, Adams grading, x range = {28}{32}, y range = {12}{15} ]
\class[rectangle](28,14)
\class[rectangle](29,13)
\class[rectangle](30,12)
\class[rectangle](30,15)
\class[rectangle](31,14)
\class[rectangle](32,13)
\structline(29,13)(31,14)

\class(27,12)
\structline(27,12)(29,13)

\class(33,15)
\structline(31,14)(33,15)
\end{sseqdata}
\printpage[ name = Y_ASS_2, page = 2 ]
\caption{The bigraded homotopy groups $\pi_{t - s,t}(C\tau \otimes \widetilde{Y})$. The line is the bottom edge of the band of \Cref{Y_vanish}.}\label{Y_ASS_ss2}
\end{figure}

We will show that all permanent cycles of $\widetilde{Y}_3$ in the range $$\frac{1}{2}(t - s) - 2.5 \le s < \frac{1}{2}(t - s) - 1.5$$ with $t - s \ge 29$ are eventual boundaries. In particular, we will show that these permanent cycles lift to $\tau^3$-torsion classes of $\pi_{t - s,t}\widetilde{Y}_3$. The groups $\pi_{t - s,t}(C\tau \otimes \widetilde{Y})$ are 0 in this range, so exactness implies that all classes we have to consider are in the image of $\pi_{t - s,t}(\Sigma^{1,2}C\tau \otimes \widetilde{Y}_2)$.

As shown in \Cref{Y3_ASS_ss2}, the groups we have to consider are the potential copies of $\mathbb{F}_2$ at $(2k,k - 2)$ and at $(2k + 1,k - 2)$ for $k \ge 14$. Figures \ref{Y2_ASS_ss2}, \ref{Y3_ASS_ss2}, and \ref{Y_ASS_ss2} illustrate this in the case $k = 15$.

First, consider the potential nonzero permanent cycle $b \in \pi_{2k,2k + (k - 2)}(C\tau \otimes \widetilde{Y}_3)$. For brevity, let $i \coloneqq 2k$ and $j \coloneqq k - 2$. In Figures \ref{Y2_ASS_ss2}, \ref{Y3_ASS_ss2}, and \ref{Y_ASS_ss2}, $i = 30$ and $j = 13$.

The class $b$ must be in the image of the composite $$\pi_{i,i + j}(\Sigma^{2,4}C\tau \otimes \widetilde{Y}) \to \pi_{i,i + j}(\Sigma^{1,2}C\tau \otimes \widetilde{Y}_2) \to \pi_{i,i + j}(C\tau \otimes \widetilde{Y}_3),$$ and its preimage is uniquely determined. Since this preimage $$a \in \pi_{i,i + j}(\Sigma^{2,4}C\tau \otimes \widetilde{Y})$$ lies in the vanishing region of $\Sigma^{2,4}\widetilde{Y}$ and is a permanent cycle for degree reasons (see the proof of \Cref{Y2_vanish}), it must lift to a $\tau$-torsion class $\hat{a} \in \pi_{i,i + j}(\Sigma^{2,4}\widetilde{Y})$. Then the image of $\hat{a}$ in $\pi_{i,i + j}\widetilde{Y}_3$ is a $\tau$-torsion lift of $b$. Thus, in the modified Adams spectral sequence picture, $b$ is killed by a $d_2$.

The other potential nonzero permanent cycle in the range $$\frac{1}{2}(t - s) - 2.5 \le s < \frac{1}{2}(t - s) - 1.5$$ lies in $\pi_{2k + 1,(2k + 1) + (k - 2)}(C\tau \otimes \widetilde{Y}_3)$. If such a nonzero class $\beta \in \pi_{2k + 1,(2k + 1) + (k - 2)}(C\tau \otimes \widetilde{Y}_3)$ exists, it has a unique preimage $\alpha \in \pi_{2k + 1,(2k + 1) + (k - 2)}(\Sigma^{1,2}C\tau \otimes \widetilde{Y}_2)$.

For brevity, let $u \coloneqq 2k + 1$ and $v \coloneqq k - 2$. Then in Figures \ref{Y2_ASS_ss2}, \ref{Y3_ASS_ss2}, and \ref{Y_ASS_ss2}, $u = 31$ and $v = 13$.

We claim that $\alpha$ is a permanent cycle. We first note that $\Sigma^{1,2}\widetilde{Y}_2$ is $\tau$-complete. This is true because $\widetilde{Y} \simeq \nu(C(2) \otimes C(\eta))$ is $\tau$-complete (see \cite[Proposition A.13]{burklund2019boundaries}) and cofibers of $\tau$-complete synthetic spectra are $\tau$-complete (\Cref{lim_tau_comp}). It may help to think about this $\tau$-completeness as saying that the modified Adams spectral sequence converges to $\pi_*(C(2) \otimes C(\eta^2))$.

By the $\tau$-completeness of $\Sigma^{1,2}\widetilde{Y}_2$, in order to show that $\alpha$ is a permanent cycle, it is enough to show that it admits compatible lifts along the maps $$\pi_{u,u + v}(\Sigma^{1,2}C\tau^{i + 1} \otimes \widetilde{Y}_2) \to \pi_{u,u + v}(\Sigma^{1,2}C\tau^i \otimes \widetilde{Y}_2)$$ for $i \ge 1$. For this, it suffices to show that these maps are surjective. The $i$th map is induced by the second map in the cofiber sequence $$\Sigma^{1,2 - i}C\tau \otimes \widetilde{Y}_2 \to \Sigma^{1,2}C\tau^{i + 1} \otimes \widetilde{Y}_2 \to \Sigma^{1,2}C\tau^i \otimes \widetilde{Y}_2$$ obtained by smashing the cofiber sequence of \Cref{tau_cofib} with $\Sigma^{1,2}\widetilde{Y}_2$. By exactness, the map $$\pi_{u,u + v}(\Sigma^{1,2}C\tau^{i + 1} \otimes \widetilde{Y}_2) \to \pi_{u,u + v}(\Sigma^{1,2}C\tau^i \otimes \widetilde{Y}_2)$$ is surjective when the $\tau$-Bockstein vanishes: $$\pi_{u,u + v}(\Sigma^{1,2}C\tau^i \otimes \widetilde{Y}_2) \to \pi_{u,u + v}(\Sigma^{2,2 - i}C\tau \otimes \widetilde{Y}_2).$$

As we can see from \Cref{Y2_ASS_ss2}, this map automatically vanishes for $i \ne 2$ because its target is 0. Thus, it is enough to show that the following map ($i = 2$) vanishes: $$\pi_{u,u + v}(\Sigma^{1,2}C\tau^2 \otimes \widetilde{Y}_2) \to \pi_{u,u + v}(\Sigma^{2,0}C\tau \otimes \widetilde{Y}_2).$$ In the modified Adams spectral sequence picture, this means that the $d_3$ from $\alpha$ vanishes. Suppose this $d_3$ is nonzero. Then the class $\gamma \in \pi_{u - 1,u + v + 2}(C\tau \otimes \widetilde{Y}_2)$ killed by this map lifts to a $\tau^2$-torsion class of $\pi_{u - 1,u + v + 2}(\Sigma^{1,2}\widetilde{Y}_2)$, which goes to 0 in the homotopy of $\tau^{-1}\Sigma^{1,2}\widetilde{Y}_2$.

However, we claim that this creates a contradiction. Since $\Sigma^{1,2}\widetilde{Y}_2$ admits a $v_1$-banded vanishing line by \Cref{Y2_vanish}, the map $$F^{k - 1}\pi_{2k}(\tau^{-1}\Sigma^{1,2}\widetilde{Y}_2) \to \pi_{2k}L_{K(1)}(\tau^{-1}\Sigma^{1,2}\widetilde{Y}_2) \cong \pi_{2k}(\Sigma L_{K(1)}(C(2) \otimes C(\eta^2)))$$ is an isomorphism. By the $\tau$-completeness of $\Sigma^{1,2}\widetilde{Y}_2$, the fact that $\gamma$ is killed by the $\tau$-Bockstein implies that there is at most one surviving class contributing to $F^{k - 1}\pi_{2k}(\tau^{-1}\Sigma^{1,2}\widetilde{Y}_2)$ (the one in $(2k,k - 1)$ in \Cref{Y2_ASS_ss2}). Hence, $$|\pi_{2k}L_{K(1)}(\Sigma C(2) \otimes C(\eta^2))| \le 2.$$

Moreover, we know from \Cref{hom_moore} that \begin{align*}
|\pi_{2k}L_{K(1)}(C(2) \otimes C(\eta^3))| &= |\pi_{2k}L_{K(1)}(C(2) \oplus \Sigma^4C(2))| \\
&= |\pi_{2k}L_{K(1)}(C(2))||\pi_{2k - 4}L_{K(1)}(C(2))| \\
&= 8
\end{align*}
From the computation used to prove \Cref{Y_vanish}, we know that $$|\pi_{2k}L_{K(1)}(C(2) \otimes C(\eta))| = 2.$$ Since $2 \cdot 2 < 8$, this contradicts the exactness of the sequence we get by applying the exact functor $L_{K(1)} \circ \tau^{-1}$ to the cofiber sequence (\ref{cofib_Y3_1}): $$\pi_{2k}L_{K(1)}(\Sigma C(2) \otimes C(\eta^2)) \to \pi_{2k}L_{K(1)}(C(2) \otimes C(\eta^3)) \to \pi_{2k}L_{K(1)}(C(2) \otimes C(\eta)).$$

Therefore, the $d_3$ vanishes, and $\alpha$ is a permanent cycle.

Since $\alpha$ is in the vanishing region of $\Sigma^{1,2}\widetilde{Y}_2$ given by \Cref{Y2_vanish}, we can conclude that $\alpha$ lifts to $\tau^2$-torsion as we did with $a$. Since $\alpha$ maps to $\beta$, $\beta$ must also lift to $\tau$-torsion.

We finally have that $\widetilde{Y}_3$ admits a $v_1$-banded vanishing line with parameters $(-1.5 \le 1,29,0.2,6.2,3)$. Since $\widetilde{Y}_3 \simeq C(\widetilde{2}) \oplus \Sigma^{4,6}C(\widetilde{2})$, $C(\widetilde{2})$ admits a $v_1$-banded vanishing line with parameters $(-1.5 \le 1,25,0.2,5,3)$. Since $C(\widetilde{2}) \simeq \nu C(2)$ (see \cite[Lemma 15.4]{burklund2019boundaries}), this encodes a $v_1$-banded vanishing line in the Adams spectral sequence for $C(2)$ and proves the theorem.
\end{proof}

\subsection{The homotopy groups of the $K(1)$-local Moore spectrum}
Our proof of \Cref{mahowald_thm_better} shows that there is a $v_1$-banded vanishing line for $C(\tilde{2})$, where the bottom edge is the line $s = \frac{1}{2}(t - s) - 1.5$. By definition, this means that the following composite is an isomorphism for all sufficiently high $i$: $$F^{\frac{1}{2}i - 1.5}\pi_iC(2) \xhookrightarrow{} \pi_iC(2) \to \pi_iL_{K(1)}C(2).$$ Since we know what the Adams spectral sequence for $C(2)$ looks like (\Cref{C2_ASS}), we can read off the homotopy groups of $L_{K(1)}C(2)$.

As with the homotopy of the $K(1)$-local sphere, the following result is well-known to experts. However, we include the following proof because a proof is not readily available in the literature and because this particular proof follows immediately from \Cref{mahowald_thm_better}.

\begin{theorem}
The homotopy groups of $L_{K(1)}C(2)$ are the following:
\begin{center}
\begin{tabular}{|c|c|}
\hline
$i \pmod{8}$ & $\pi_i(L_{K(1)}C(2))$ \\
\hline
0 & $\mathbb{Z}/2 \oplus \mathbb{Z}/2$ \\
\hline
1 & $\mathbb{Z}/2 \oplus \mathbb{Z}/4$ \\
\hline
2 & $\mathbb{Z}/2 \oplus \mathbb{Z}/4$ \\
\hline
3 & $\mathbb{Z}/2 \oplus \mathbb{Z}/2$ \\
\hline
4 & $\mathbb{Z}/2$ \\
\hline
5 & 0 \\
\hline
6 & 0 \\
\hline
7 & $\mathbb{Z}/2$ \\
\hline
\end{tabular}
\end{center}
\end{theorem}

\begin{proof}
We first recall that $C(2)$ has a $v_1^4$ self-map (see \cite{adams1966groups}), so that the self-map $$L_{K(1)}v_1^4: \Sigma^8L_{K(1)}C(2) \to L_{K(1)}C(2)$$ is an equivalence. In particular, this map induces an isomorphism $$\pi_{* - 8}L_{K(1)}C(2) \cong \pi_*L_{K(1)}C(2),$$ so the homotopy groups of $L_{K(1)}C(2)$ have period 8.

We already know $\pi_iL_{K(1)}C(2)$ for $i \equiv 4,5,6,7 \pmod{8}$ since the orders of these groups uniquely determine the groups. We know $\pi_0L_{K(1)}C(2) \cong \mathbb{Z}/2 \oplus \mathbb{Z}/2$ from the proof of \Cref{hom_moore}, so our observation about $v_1^4$-periodicity implies that $\pi_iL_{K(1)}C(2) \cong \mathbb{Z}/2 \oplus \mathbb{Z}/2$ for all $i \equiv 0 \pmod{8}$.

We only have the following extension problems left to solve:
$$0 \to \mathbb{Z}/2 \oplus \mathbb{Z}/2 \to \pi_{8k + 1}L_{K(1)}C(2) \to \mathbb{Z}/2 \to 0$$
$$0 \to \mathbb{Z}/2 \to \pi_{8k + 2}L_{K(1)}C(2) \to \mathbb{Z}/2 \oplus \mathbb{Z}/2 \to 0$$
$$0 \to \mathbb{Z}/2 \to \pi_{8k + 3}L_{K(1)}C(2) \to \mathbb{Z}/2 \to 0$$
We address them in order:
\begin{itemize}
\item $8k + 1$: Pick $k$ large enough that we have an isomorphism $$F^{\frac{1}{2}(8k + 1) - 1.5}\pi_{8k + 1}C(2) \xhookrightarrow{} \pi_{8k + 1}C(2) \to \pi_{8k + 1}L_{K(1)}C(2).$$ Then as we can see from \Cref{C2_ASS}, there is a 2-extension in $F^{\frac{1}{2}(8k + 1) - 1.5}\pi_{8k + 1}C(2)$. Thus, $\pi_{8k + 1}L_{K(1)}C(2) \cong \mathbb{Z}/2 \oplus \mathbb{Z}/4$. By $v_1^4$-periodicity, this determines $\pi_{8k + 1}L_{K(1)}C(2)$ for all $k$.

\item $8k + 2$: This proceeds identically to the case $8k + 1$. We use the 2-extension in \Cref{C2_ASS} to deduce that $\pi_{8k + 2}L_{K(1)}C(2) \cong \mathbb{Z}/2 \oplus \mathbb{Z}/4$.

\item $8k + 3$: Again, pick $k$ large. Then looking at \Cref{C2_ASS}, we want to rule out the possibility of a hidden 2-extension in $F^{\frac{1}{2}(8k + 3) - 1.5}\pi_{8k + 3}C(2)$. Our analysis in the case $8k + 2$ shows that there is no hidden 2-extension between the bottom two classes in $F^{\frac{1}{2}(8k + 2) - 1.5}\pi_{8k + 2}C(2)$ because this would contradict the 2-extension above it. The $\eta$-multiplication from column $8k + 2$ to column $8k + 3$ then rules out the possibility that a 2-extension occurs in $F^{\frac{1}{2}(8k + 3) - 1.5}\pi_{8k + 3}C(2)$. We conclude that $\pi_{8k + 3}L_{K(1)}C(2) \cong \mathbb{Z}/2 \oplus \mathbb{Z}/2$.
\end{itemize}
\end{proof}

\bibliographystyle{plain}
\nocite{*}
\bibliography{Bibliography}

\end{document}